\documentclass[11pt]{amsart}

\usepackage{amsfonts}
\usepackage{amscd}
\usepackage{epsfig}
\usepackage{epic,eepic}
\usepackage{graphicx}
\usepackage[all]{xy}
\usepackage{xypic}
\usepackage{psfrag}
\usepackage{amsmath}
\usepackage{amsthm}
\usepackage{setspace}
\usepackage{here}

\theoremstyle{plain}
\newtheorem{theorem}{Theorem}[section]
\newtheorem{lemma}[theorem]{Lemma}

\newtheorem{corollary}[theorem]{Corollary}

\theoremstyle{definition}
\newtheorem{remark}[theorem]{Remark}

\newtheorem{example}[theorem]{Example}

\newcommand{\Mg}{\overline{\mathcal M}_{g}}
\newcommand{\PP}{\mathbb P}
\newcommand{\Mof}{\overline{M}_{0,4}}
\newcommand{\Hdg}{\mathcal H_d({\bf c})}
\newcommand{\BHdg}{\overline{\mathcal H}_d({\bf c})}
\newcommand{\Ndg}{N_d({\bf c})}
\newcommand{\gof}{\gamma_1, \gamma_2, \gamma_3, \gamma_4}
\newcommand{\prodgof}{\gamma_1\gamma_2\gamma_3\gamma_4}
\newcommand{\OO}{\mathcal O}
\newcommand{\QQ}{\mathcal Q}
\newcommand{\HH}{\mathcal H}
\newcommand{\ZZ}{\overline{Z}_{\OO}}
\newcommand{\RR}{\mathbb R}

\title{Covers of the projective line and the
moduli space of quadratic differentials}

\author{Dawei Chen}

\address{Department of Mathematics, Statistics and Computer Science, University of Illinois at
  Chicago, Chicago, IL 60607}

\email{dwchen@math.uic.edu}

\begin{document}
\bibliographystyle{plain}

\begin{abstract}
Consider the 1-dimensional Hurwitz space parameterizing covers of $\PP^1$ branched at four points. We study its intersection with divisor classes on the moduli space of curves. As an application, we calculate the slope of the Teichm\"{u}ller curve parameterizing square-tiled cyclic covers and recover the sum of its Lyapunov exponents obtained by Forni, Matheus and Zorich \cite{FMZ}. Motivated by the work of Eskin, Kontsevich and Zorich \cite{EKZ1}, we exhibit a relation among the slope of Hurwitz spaces, the sum of Lyapunov exponents and the Siegel-Veech constant for the moduli space of quadratic differentials. 
\end{abstract}

\maketitle
\tableofcontents

\section{Introduction}
Let $\Mg$ denote the moduli space of stable genus $g$ curves. The geometry of $\Mg$ can be revealed by studying 1-dimensional families of genus $g$ curves. A useful construction of such families arises from branch covers. For instance, Cornalba and Harris \cite{CH} studied 1-dimensional families of hyperelliptic curves, which have very special numerical classes. Stankova \cite{S} generalized the result to families of trigonal curves. Along a slightly different direction, Harris and Morrison \cite{HM1} studied 1-dimensional families of degree $d\gg 0$, simply branched covers of $\PP^1$, via which they obtained lower bounds for the slope of effective divisors on $\Mg$. In \cite{C1}, families of covers of elliptic curves were studied for the same purpose. In addition, 1-dimensional families of covers of elliptic curves with a unique branch point, also called arithmetic Teichm\"{u}ller curves, have significance in complex dynamics. In \cite{C2}, a relation among their slopes and the sum of Lyapunov exponents of the Hodge bundle on the moduli space of Abelian differentials was established. 

In this paper, we consider covers of $\PP^1$ branched at four points of arbitrary ramification type. Varying a branch point, we get a 1-dimensional Hurwitz space parameterizing such covers. Let us first summarize the main results. In Theorem~\ref{monodromy}, we give a criterion of distinguishing irreducible components of the Hurwitz space. In Theorem~\ref{slope}, we analyze its intersection with divisor classes on $\Mg$ and derive a slope formula. We apply the formula to the Hurwitz space of cyclic covers in Theorem~\ref{cyclic} and recover the sum of 
its Lyapunov exponents in Theorem~\ref{Lcyclic}, which was first obtained by Forni, Matheus and Zorich \cite{FMZ}. Building on the work of Eskin, Kontsevich and Zorich \cite{EKZ1}, we find a relation among the slope, the Siegel-Veech constant and the sum of Lyapunov exponents for the moduli space of quadratic differentials. Aiming for understanding the effective cone of $\Mg$, at the end we discuss some interesting interplay between $\Mg$ and the moduli space of quadratic differentials. Throughout the paper, we work over the complex number field $\mathbb C$. 

Let $p_1, p_2, p_3, p_4$ be four distinct points on $\PP^1$ and $c_1, c_2, c_3, c_4$ four conjugacy classes of the permutation group $S_d$ such that the ramification type over $p_i$ corresponds to $c_i$ for $1\leq i\leq 4$. Fix the ramification profile
$${\bf c} = (c_1, c_2, c_3, c_4).$$ 
Suppose $c_i$ consists of $k_i$ cycles, each of which has length $a_{i,j}$ for $1\leq j \leq k_i$, where $\sum_{j=1}^{k_i} a_{i,j} = d$, $1\leq i\leq 4$. The genus $g$ of connected covers with the ramification profile ${\bf c}$  is determined by the Riemann-Hurwitz formula 
$g= d + 1 - \frac{1}{2}\sum_{i=1}^4 k_i.$ 

Fix $p_1, p_2, p_3$ and vary $p_4$ along $\PP^1$. We obtain a 1-dimensional Hurwitz space $\Hdg$ of such covers. 
When $p_4$ meets $p_i$, $1\leq i\leq 3$, we get degenerate covers between nodal curves in the sense of admissible covers. Let $\BHdg$ denote the compactification of $\Hdg$ parameterizing degree $d$, connected admissible covers with the ramification profile ${\bf c}$, cf. \cite[3.G]{HM2}. There are two natural morphisms 
$$\xymatrix{
\BHdg \ar[r]^{h} \ar[d]^{e}  &  \Mg  \\
 \Mof            & }
$$ 
The morphism $h$ sends a genus $g$ cover to the stabilization of its domain curve parameterized in $\Mg$. The morphism $e$ sends a cover to its image rational curve marked at the four branch points paramterized in $\Mof\cong \PP^1$. Note that $e$ restricted to $\Hdg$ is an unramified cover of degree $\Ndg$, where $\Ndg$ equals the number of non-isomorphic covers of
$\PP^1$ with four fixed branch points and the ramification profile ${\bf c}$. 

Let $\pi_1(\PP^1, p_1, p_2, p_3, p_4 ; b)$ denote the fundamental group of $\PP^1$ punctured at the four branch points with a base point $b$. 
Let $\alpha_i$ denote a closed oriented path circular around $p_i$ for $1\leq i\leq 4$ such that $\alpha_1\alpha_2\alpha_3\alpha_4 = id$ in $\pi_1(\PP^1, p_1, p_2, p_3, p_4 ; b)$. See the following picture.

\begin{figure}[H]
    \centering
    \psfrag{p1}{$p_{1}$}
    \psfrag{p2}{$p_{2}$}
    \psfrag{p3}{$p_{3}$}
    \psfrag{p4}{$p_{4}$}
    \psfrag{b}{$b$}
    \psfrag{a1}{$\alpha_{1}$}
    \psfrag{a2}{$\alpha_{2}$}
    \psfrag{a3}{$\alpha_{3}$}
    \psfrag{a4}{$\alpha_{4}$}
    \includegraphics[scale=0.5]{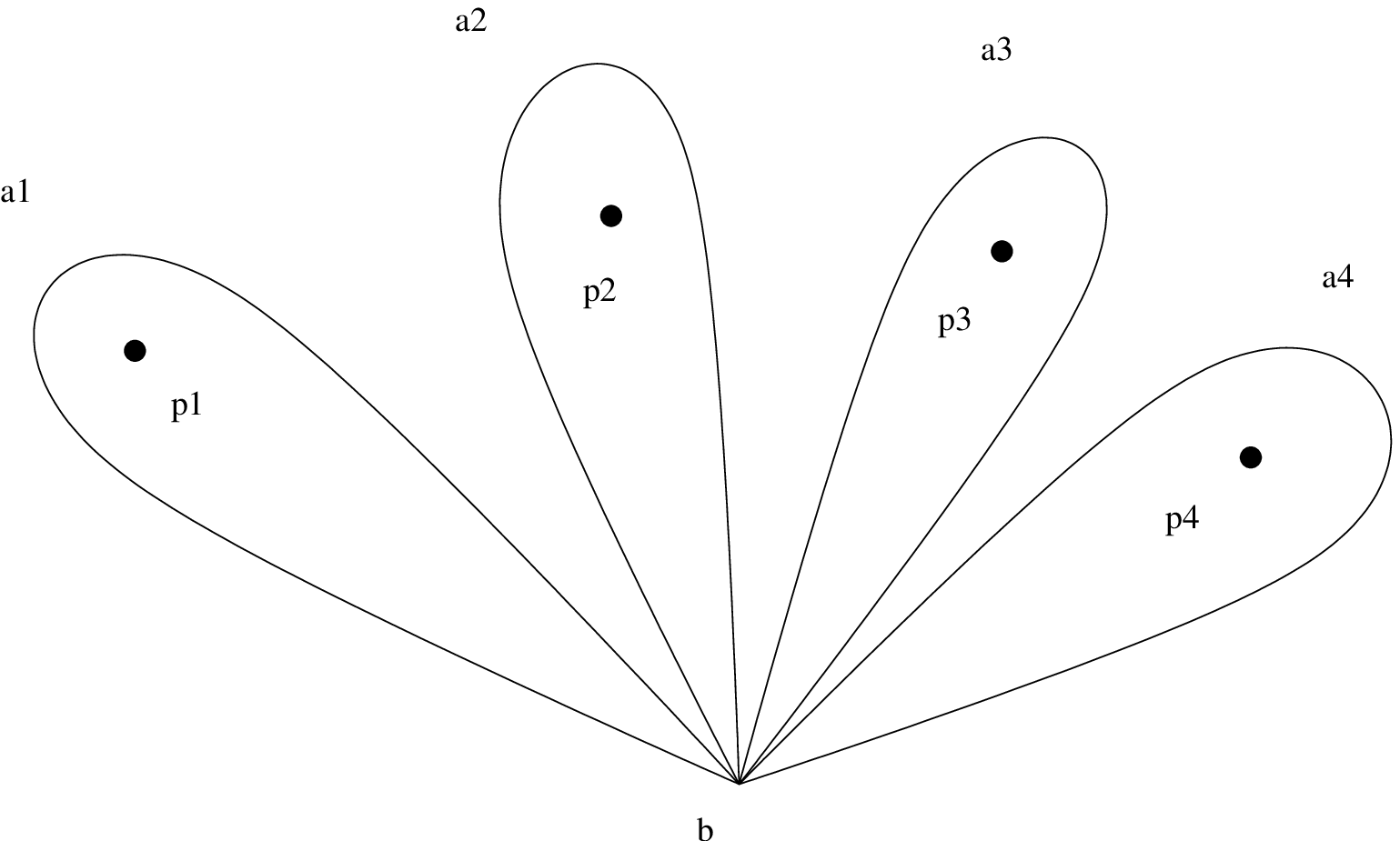}
    \end{figure}

Label the $d$ sheets of a cover by $1,\ldots, d$. Then a cover parameterized in $\Hdg$ corresponds to an element in $\mbox{Hom}(\pi_1(\PP^1, p_1, p_2, p_3, p_4 ; b), S_d)$. 
Let $\gamma_i$ denote the monodromy image of $\alpha_i$ in $S_d$ for $1\leq i \leq 4$ and call ${\bf r} = (\gof)$ the monodromy data of the corresponding cover.  
Define the following set of equivalence classes
$$Cov_d({\bf c}) = \{ {\bf r}=  (\gof)\ |\ \gamma_i \in c_i, \ \prodgof = id,  $$
$$  \langle \gof \rangle \ \mbox{is a transitive subgroup of}\ S_d   \}/\sim. $$
The equivalence relation $\sim$ is defined for two data $\bf r \sim \bf r'$ if there exists $\tau\in S_d$ such that 
$$\tau (\gof) \tau^{-1} = (\gamma_1', \gamma_2', \gamma_3', \gamma_4').$$ 
Two covers corresponding to ${\bf r}$ and ${\bf r'}$ are isomorphic if and only if ${\bf r} \sim {\bf r'}$, where the conjugate action by $\tau$ amounts to relabeling the $d$ sheets of a cover. The transitivity guarantees that the covers are connected. Hence, the set of non-isomorphic covers of $\PP^1$ with four fixed branch points and the ramification profile ${\bf c}$ can be identified as $Cov_d({\bf c}) $. Namely, a fiber of the finite morphism
$e: \Hdg\rightarrow M_{0,4} \cong \PP^1\backslash\{p_1,p_2,p_3 \}$ is parameterized by $Cov_d({\bf c}) $ and $e$ has degree $\Ndg = |Cov_d({\bf c}) |.$ 

The Hurwitz space $\BHdg$ may be reducible. Since $e$ restricted to $\Hdg$ is unramified, the fundamental group $\pi_1(M_{0,4} ; b')$ acts on $Cov_d({\bf c}) $ and each orbit corresponds to an irreducible component of $\Hdg$. Let $\beta_1$ and $\beta_2$ denote two closed oriented paths around $p_3$ and $p_2, p_3$,  respectively, as in the following picture. 

\begin{figure}[H]
    \centering
    \psfrag{p1}{$p_{1}$}
    \psfrag{p2}{$p_{2}$}
    \psfrag{p3}{$p_{3}$}
    \psfrag{b}{$b'$}
    \psfrag{a1}{$\beta_{1}$}
    \psfrag{a2}{$\beta_{2}$}
    \includegraphics[scale=0.5]{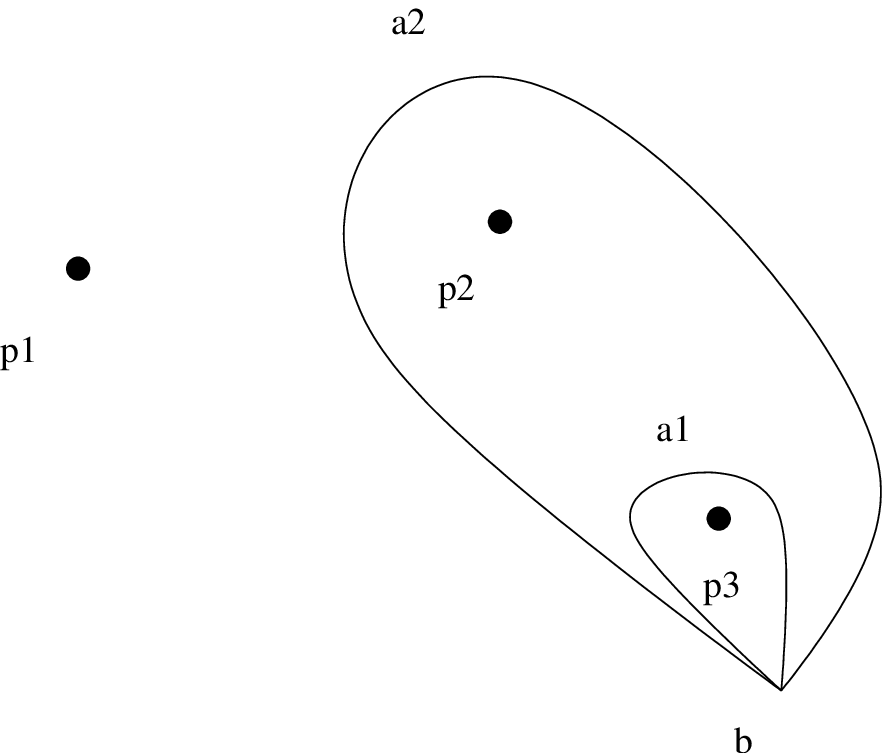}
    \end{figure}

Let $g_i$ denote the action on $Cov_d({\bf c}) $ induced by $\beta_i$ for $i = 1,2$. 

\begin{theorem}
\label{monodromy}
Let ${\bf r} = (\gof)$ be a representative of an equivalence class in $Cov_d({\bf c}) $. We have
$$ g_1 ({\bf r}) = (\gamma_1, \gamma_2, \gamma_4^{-1} \gamma_3 \gamma_4, (\gamma_3\gamma_4)^{-1} \gamma_4(\gamma_3\gamma_4)), $$
$$g_2({\bf r}) = (\gamma_1, \gamma_4^{-1}\gamma_2\gamma_4, \gamma_4^{-1}\gamma_3\gamma_4, (\gamma_2\gamma_3\gamma_4)^{-1} \gamma_4(\gamma_2\gamma_3\gamma_4)). $$
Two covers are parameterized in the same irreducible component of $\Hdg$ if and only if their monodromy data are in the same orbit under the actions generated by $g_1, g_2$. 
\end{theorem}

\begin{proof}
Recall that $\alpha_i$ denotes a closed oriented path around $p_i$ for $1\leq i\leq 4$ such that $\alpha_1\alpha_2\alpha_3\alpha_4 = id$ in $\pi_1(\PP^1, p_1, p_2, p_3, p_4 ; b)$.

Let the moving branch point $p_4$ go around the path $\beta_1$ once. The resulting paths have the following expression 
$$(\alpha'_1, \alpha'_2, \alpha'_3, \alpha'_4) = (\alpha_1, \alpha_2, \alpha_4^{-1} \alpha_3 \alpha_4, (\alpha_3\alpha_4)^{-1} \alpha_4(\alpha_3\alpha_4)) $$
by the original paths, hence the corresponding action $g_1$ on the monodromy data $(\gof)$ has the same expression. See the following picture. 

\begin{figure}[H]
    \centering
    \psfrag{p1}{$p_{1}$}
    \psfrag{p2}{$p_{2}$}
    \psfrag{p3}{$p_{3}$}
    \psfrag{p4}{$p_{4}$}
    \psfrag{b}{$b$}
    \psfrag{a1}{$\alpha'_{1}$}
    \psfrag{a2}{$\alpha'_{2}$}
    \psfrag{a3}{$\alpha'_{3}$}
    \psfrag{a4}{$\alpha'_{4}$}
    \psfrag{c1}{$\beta_1$}
    \includegraphics[scale=0.5]{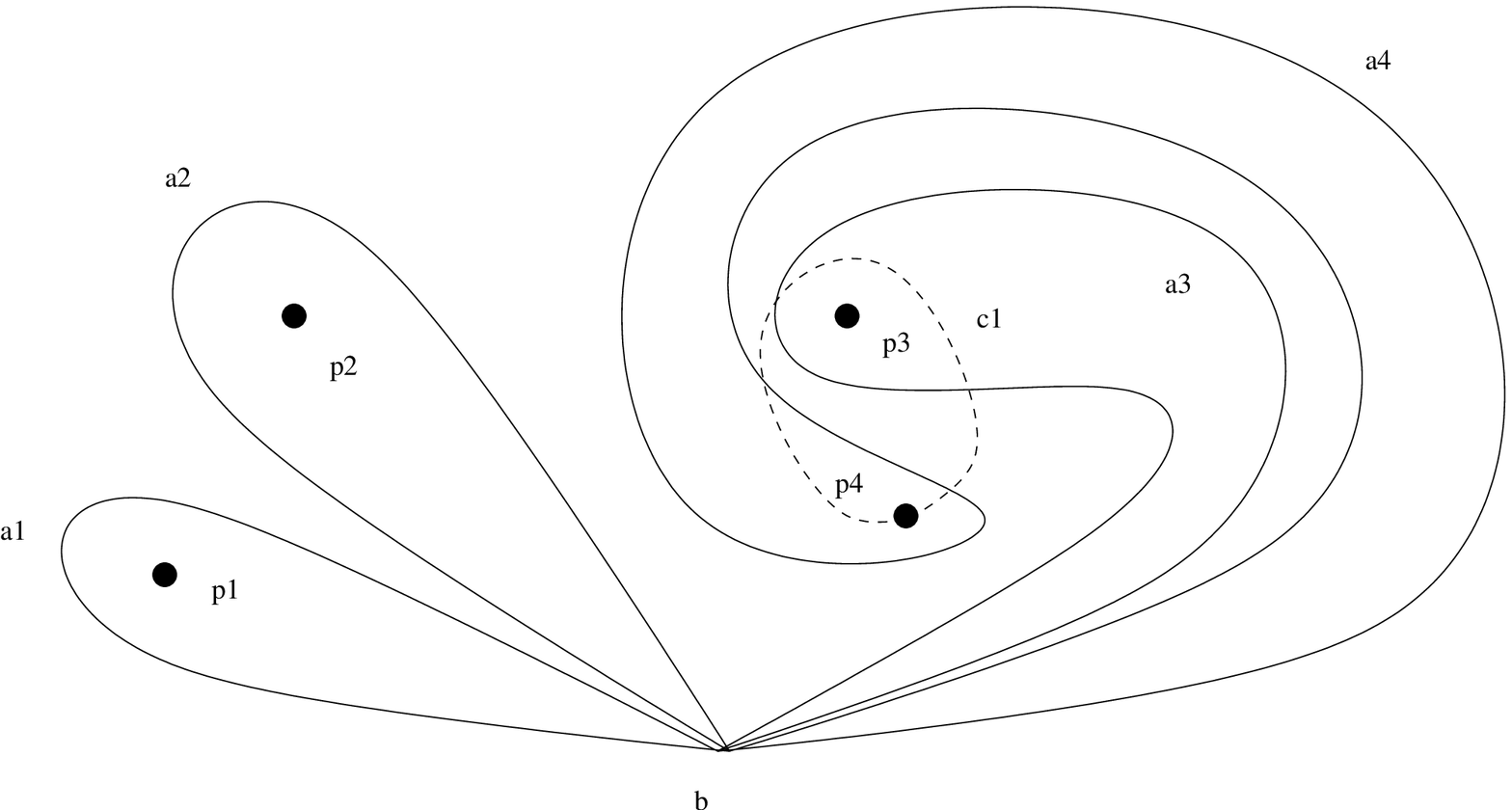}
    \end{figure}

Let the moving branch point $p_4$ go around the path $\beta_2$ once. The resulting paths have the following expression 
$$(\alpha'_1, \alpha'_2, \alpha'_3, \alpha'_4) = (\alpha_1, \alpha_4^{-1}\alpha_2\alpha_4, \alpha_4^{-1}\alpha_3\alpha_4, (\alpha_2\alpha_3\alpha_4)^{-1} \alpha_4(\alpha_2\alpha_3\alpha_4)) $$
by the original paths, hence the corresponding action $g_2$ on the monodromy data $(\gof)$ has the same expression. See the following picture. 

\begin{figure}[H]
    \centering
    \psfrag{p1}{$p_{1}$}
    \psfrag{p2}{$p_{2}$}
    \psfrag{p3}{$p_{3}$}
    \psfrag{p4}{$p_{4}$}
    \psfrag{b}{$b$}
    \psfrag{a1}{$\alpha'_{1}$}
    \psfrag{a2}{$\alpha'_{2}$}
    \psfrag{a3}{$\alpha'_{3}$}
    \psfrag{a4}{$\alpha'_{4}$}
    \psfrag{c2}{$\beta_2$}
    \includegraphics[scale=0.5]{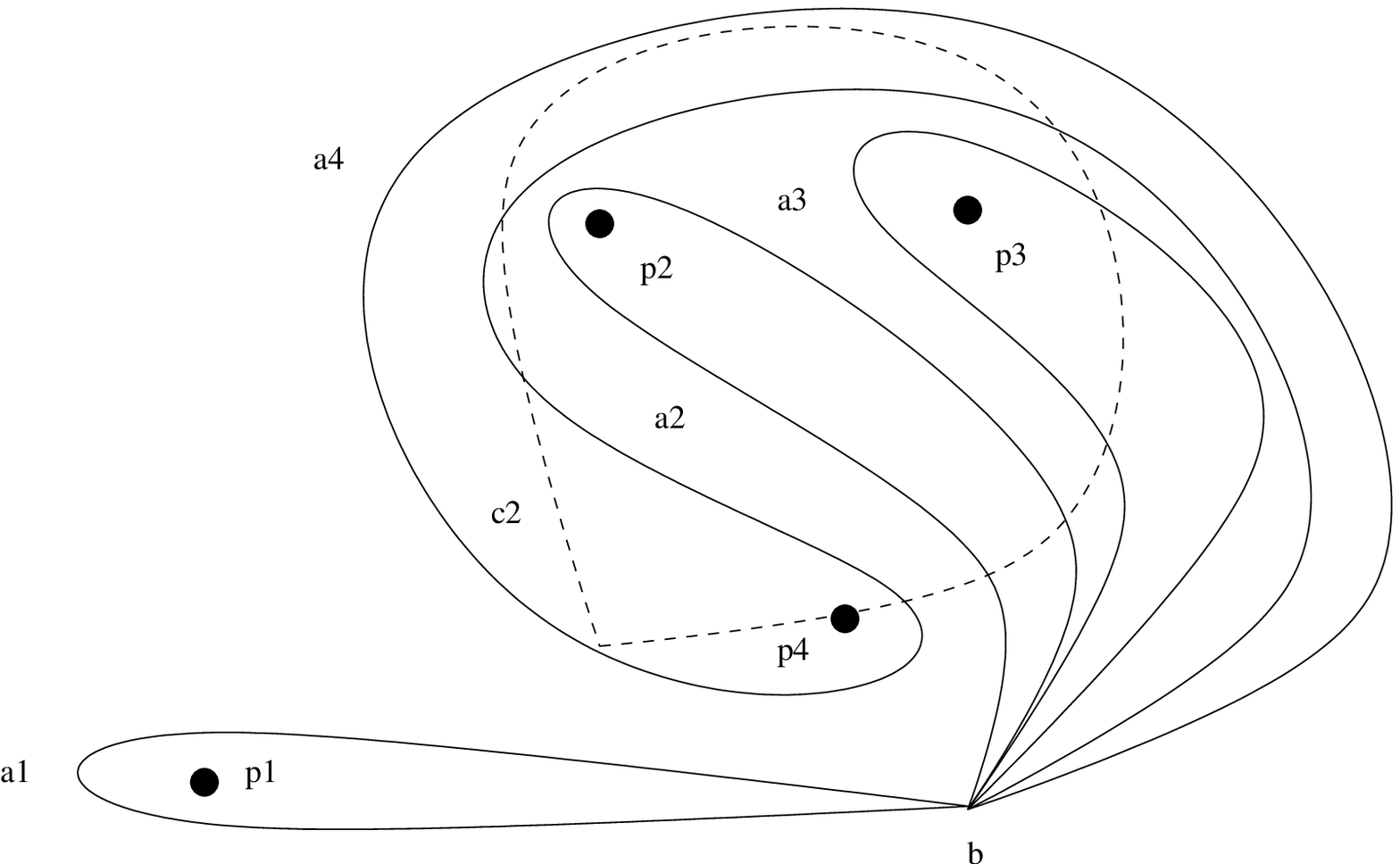}
    \end{figure}

Since $\Hdg\rightarrow M_{0,4}\cong\PP^1 \backslash \{p_1, p_2, p_3\}$ is unramified and $\pi_1(\PP^1, p_1, p_2, p_3 ; b')$ is generated by $\beta_1, \beta_2$, the irreducible components 
of $\Hdg$ correspond to the orbits of  $Cov_d({\bf c}) $ under the actions generated by $g_1, g_2$. Note that these actions are well-defined with respect to the equivalence relation $\sim$. 
\end{proof}

Let $\OO\subset Cov_d({\bf c}) $ denote an orbit of the above actions. 
Let $Z_{\OO}$ be the corresponding irreducible component of $\Hdg$ and $\ZZ$ be its closure 
in $\BHdg$. The morphism $h$ maps $\ZZ$ to an irreducible curve in $\Mg$. Let $\lambda$ denote the first Chern class of the Hodge bundle and $\delta$ be the total boundary class of $\Mg$. Define the slope of $\ZZ$ as 
$$s(\ZZ) = \frac{\mbox{deg} \ h^{*}\delta |_{\ZZ}}{\mbox{deg} \ h^{*}\lambda |_{\ZZ}}.$$ 
The slope is invariant under a finite base change. 

In order to calculate $s(\ZZ)$, we need to analyze singular admissible covers that arise in $\ZZ$. Suppose a smooth cover $\pi$ corresponding to ${\bf r}$ degenerates to a singular cover $\pi_0$ when the moving branch point $p_4$ approaches $p_3$. The image of $\pi_0$ is a nodal union of two smooth rational curves $Q_{12}\cup_{p_0} Q_{34}$ with $p_0, p_1, p_2$ on $Q_{12}$ and $p_0, p_3, p_4$ on $Q_{34}$. Locally there is a vanishing cycle $\alpha_0$ that shrinks to the node $p_0$. See the following picture. 

\begin{figure}[H]
    \centering
    \psfrag{p1}{$p_{2}$}
    \psfrag{p2}{$p_{1}$}
    \psfrag{p3}{$p_{3}$}
    \psfrag{p4}{$p_{4}$}
    \psfrag{a1}{$\alpha_{2}$}
    \psfrag{a2}{$\alpha_{1}$}
    \psfrag{a3}{$\alpha_{3}$}
    \psfrag{a4}{$\alpha_{4}$}
    \psfrag{a0}{$\alpha_0$}
    \psfrag{Q12}{$Q_{12}$}
    \psfrag{Q34}{$Q_{34}$}
    \includegraphics[scale=0.5]{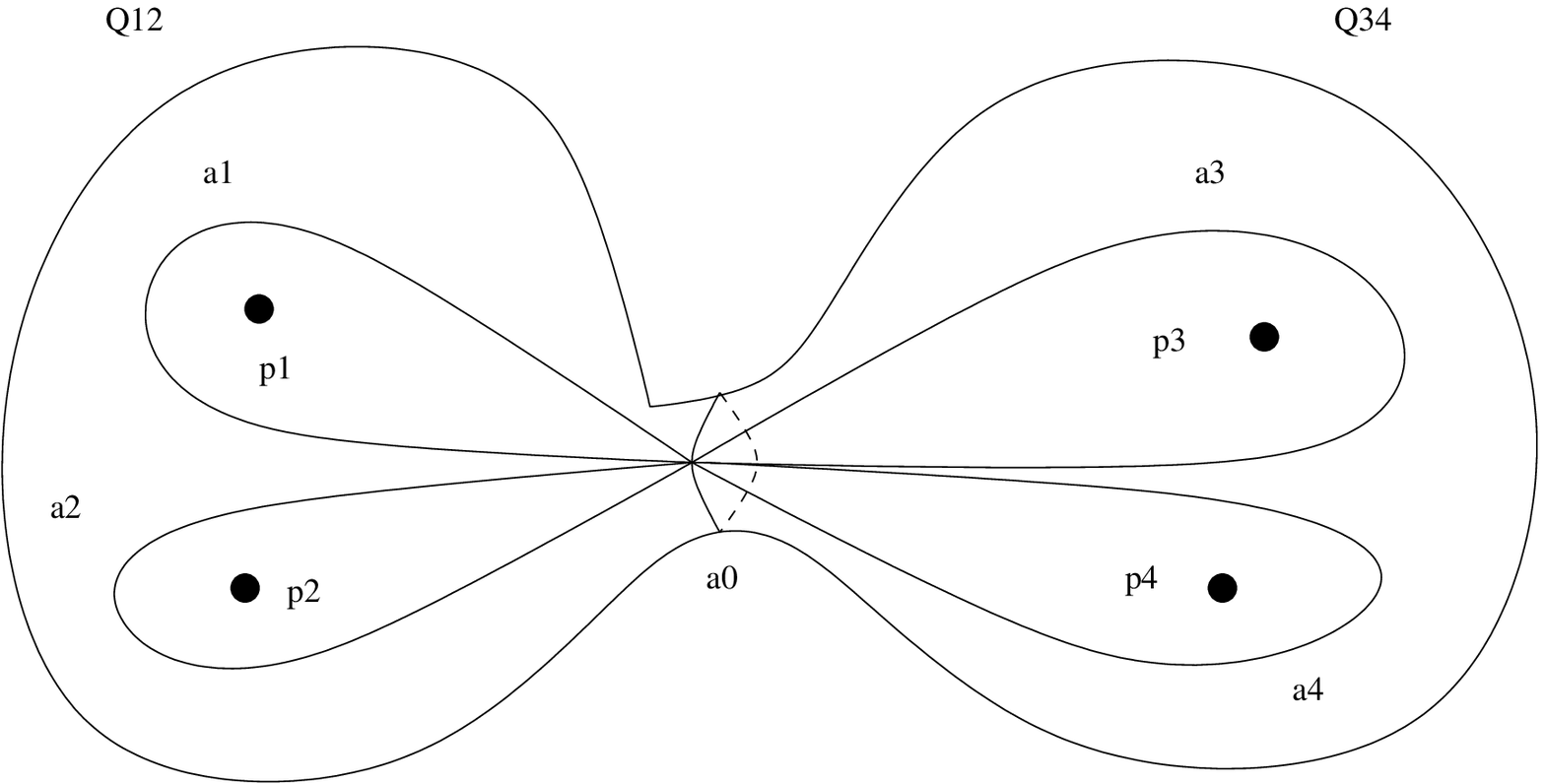}
    \end{figure}

Let $\gamma_0$ be the monodromy image of $\alpha_0$ in $S_d$. We have $\gamma_0 = (\gamma_1\gamma_2)^{-1} = \gamma_3\gamma_4$. Suppose $\gamma_0$ consists of $m$ cycles of length $a_1,\ldots, a_m$. Then the domain curve $C_0$ of $\pi_0$ has $m$ nodes $q_i$ 
for $1\leq i \leq m$ and $\pi_0$ restricted to a neighborhood of $q_i$ maps like $(x_i,y_i)\rightarrow (u,v) = (x_i^{a_i}, y_i^{a_i})$, where $(x_i,y_i)$ and $(u,v)$ are the local coordinates of the two branches of the nodes $q_i$ and $p_0$, respectively. Let $C_{12}, C_{34}$ be the two components of $C_0$ corresponding to the pre-images of $Q_{12}, Q_{34}$, respectively. 
The restriction of $\pi_0$ to $C_{12}$ is a degree $d$ cover of $Q_{12}$ branched at $p_0, p_1, p_2$ with the monodromy data $(\gamma_0, \gamma_1, \gamma_2)$ as an element in 
$\mbox{Hom}(\pi_1(Q_{12}, p_0, p_1, p_2; b), S_d)$. The connected components of $C_{12}$ correspond to the orbits of the $d$ letters $\{1,\ldots, d\}$ under the permutations generated by 
 $\gamma_0, \gamma_1, \gamma_2$. The same analysis holds for $\pi_0$ restricted to $C_{34}$. In particular, the topological type of $\pi_0$ is uniquely determined by the monodromy data 
 ${\bf r} = (\gof)$ of the nearby smooth cover $\pi$.  
 
 Let $C^{st}_0$ be the stabilization of $C_0$ by blowing down unstable rational components. Define $\delta(q_i)= 1$ or $0$, depending on whether or not the node $q_i$ maps to a node 
 of $C^{st}_0$ via $C_0 \rightarrow C_0^{st}$. Associate to ${\bf r}$ the following two weights
 $$ \delta_{3} ({\bf r}) = \sum_{i=1}^{m} \frac{1}{a_i}\delta(q_i)\ \mbox{and}\ \delta'_{3} ({\bf r}) = \sum_{i=1}^m\frac{1}{a_i}. $$
We similarly define $ \delta_{1} ({\bf r}), \delta'_{1} ({\bf r}) $ and $\delta_{2} ({\bf r}), \delta'_{2} ({\bf r})$ when $p_4$ approaches $p_1$ and $p_2$, respectively.  
Summing up the weights over all ${\bf r}$ in the orbit $\OO$, we define 
$$ \delta_{\OO} = \sum_{{\bf r}\in \OO} \sum_{i=1}^3\delta_i ({\bf r})\ \mbox{and}\ \delta'_{\OO} = \sum_{{\bf r}\in \OO} \sum_{i=1}^3\delta_i' ({\bf r}). $$

\begin{theorem}
\label{slope}
The slope of $\ZZ$ has the following expression 
$$ s(\ZZ) = \frac{12\delta_{\OO}}{\delta'_{\OO}+(d-\sum_{i=1}^4\sum_{j=1}^{k_i}\frac{1}{a_{i,j}})|\OO|},$$
where $|\OO|$ denotes the cardinality of the orbit $\OO$. 
\end{theorem}

\begin{proof}
Recall how a nodal cover $\pi_0$ arises as the limit of a smooth cover $\pi$ when
the moving branch point $p_4$ approaches $p_3$. Suppose locally around a node of the domain curve $C_0$, $\pi_0$ maps like $(x ,y)\rightarrow (u,v) = (x^{a}, y^{a})$. If this node maps to a node of the stabilization $C_0^{st}$, it contributes $\frac{1}{a}$ to the degree of the total boundary class $\delta$ restricted to $\ZZ$. The weight $\frac{1}{a}$ comes from an orbifold viewpoint: for $s = xy$, $t = uv = x^ay^a = s^a$, locally a base change of degree $a$ is needed to realize a universal covering family. If this node maps to a smooth point of $C_0^{st}$, it does not contribute to the intersection with $\delta$. Summing over all nodes of $C_0$, the limit $\pi_0$ of $\pi$ contributes $ \delta ({\bf r})$ to the restriction of $\delta$. Summing over all $\pi$ parameterized in the orbit $\OO$ and also considering when $p_4$ approaches $p_1$ or $p_2$, we get $\mbox{deg} \ h^{*}\delta |_{\ZZ} = \delta_{\OO}$ by its definition. 

Since the slope is invariant under a finite base change, we can assume there is a universal covering map (after the base change) as follows: 
$$\xymatrix{
S \ar[r]^{f} \ar[d]  &  T \ar[d]  \\
\ZZ \ar[r]^{e}           & \Mof }
$$ 
The surface $S$ is a genus $g$ fibration over $\ZZ$ parameterizing the domain curves of the covers in $\ZZ$. The surface $T$ is isomorphic to the blow-up of $\PP^1\times\PP^1$ at the intersection points of three horizontal sections with the diagonal, where the sections and the diagonal parameterize the fixed branch points $p_1, p_2, p_3$ and the moving point 
$p_4$, respectively. Let $\Sigma_i$ on $T$ denote the proper transform of the three sections and the diagonal for $1\leq i \leq 4$. The horizontal morphisms $e$ and $f$ are finite of degree $|\OO|$ and $d|\OO|$, respectively. We have 
$$ f^{*} \Sigma_i = \sum_{j=1}^{k_i} a_{i,j} \Gamma_{i,j},$$
where $\Gamma_{i,j}$ are sections of $S$ parameterizing ramification points of the covers that map to $p_i$ for $1\leq i\leq 4$. 

Let $\omega_{S/\ZZ}$ and $\omega_{T/\Mof}$ be the relative dualizing sheaves of the two families $S$ and $T$, respectively. By Riemann-Roch, we have 
$$ \omega_{S/\ZZ} = f^{*} \omega_{T/\Mof} +  \sum_{i=1}^4\sum_{j=1}^{k_i} (a_{i,j}-1) \Gamma_{i,j}. $$ 
Hence, we can compute the self-intersection 
$$ (\omega_{S/\ZZ})^2  = (f^{*} \omega_{T/\Mof})^2 + 2 \sum_{i=1}^4\sum_{j=1}^{k_i} (a_{i,j}-1)(f^{*} \omega_{T/\Mof})\ldotp \Gamma_{i,j}
 +  \Big(\sum_{i=1}^4\sum_{j=1}^{k_i} (a_{i,j}-1) \Gamma_{i,j}\Big)^2. $$
Moreover, we have
$$ (f^{*} \omega_{T/\Mof})^2 = d|\OO| (\omega_{T/\Mof})^2 = -3d |\OO|, $$
$$ (f^{*} \omega_{T/\Mof})\ldotp \Gamma_{i,j} = \omega_{T/\Mof}\ldotp (f_{*}\Gamma_{i,j}) = |\OO| (\omega_{T/\Mof}\ldotp \Sigma_i) = - |\OO| (\Sigma_i)^2 = |\OO|, $$
$$ \Gamma_{i,j}\ldotp \Gamma_{i', j'} = 0 \ \mbox{for}\ (i,j)\neq (i',j') \ \mbox{and}\ (\Gamma_{i,j})^2  = \frac{|\OO|}{a_{i,j}}(\Sigma_i)^2 = -\frac{|\OO|}{a_{i,j}}. $$ 
Using the condition $\sum_{j=1}^{k_i} a_{i,j} = d$ for $1\leq i \leq 4$, a routine calculation shows that 
$$  (\omega_{S/\ZZ})^2 = \Big(d - \sum_{i=1}^4\sum_{j=1}^{k_i}\frac{1}{a_{i,j}}\Big) |\OO| . $$

For the family $S$ over $\ZZ$, by Mumford's relation, we have $\lambda = \frac{\delta' + \kappa}{12}$ restricted to $\ZZ$, where 
$\kappa =  (\omega_{S/\ZZ})^2$ and $\delta'$ enumerates the nodes of the fibers in $S$, which equals $\delta'_{\OO}$ by definition. 
The weight $\frac{1}{a_{i,j}}$ in the definition of $\delta'$ is due to the need of a base change of degree $a_{i,j}$ to realize a universal covering map around such a node. 
Note that $\delta'_{\OO}$ might be different from $\delta_{\OO}.$ If a node of the domain curve $C_0$ of a degenerate cover $\pi_0$ maps to a smooth point of the stabilization $C_0^{st}$, it does not contribute to the intersection with the boundary class $\delta$. Nevertheless, on the surface $S$, this node remains in the fiber $C_0$ and it contributes to $\delta'$. 

Overall, we obtain 
$$ \mbox{deg}\ \delta|_{\ZZ} = \delta_{\OO} \ \mbox{and}\ \mbox{deg}\ \lambda|_{\ZZ} = \frac{\delta'_{\OO}+(d-\sum_{i=1}^4\sum_{j=1}^{k_i}\frac{1}{a_{i,j}})|\OO|}{12}. $$
Hence, the slope of $\ZZ$ has the desired expression.  
\end{proof}

To illustrate how to apply Theorems~\ref{monodromy} and \ref{slope}, let us consider the following example. 

\begin{example}
\label{application}
Consider the case when all the $c_i$'s are of length $d$ for $1\leq i \leq 4$, where $d\geq 3$ is an odd number. There is an orbit $\OO_1$ of $Cov_d({\bf c})$ that consists of a unique 
monodromy data ${\bf r} = (\gamma, \gamma^{-1}, \gamma, \gamma^{-1})$ (up to equivalence) under the actions generated by $g_1, g_2$ in Theorem~\ref{monodromy}, where 
$\gamma$ is any cycle of length $d$ in $S_d$. Hence, we have the cardinality $|\OO_1| = 1$. The genus $g$ of the covers equals $d-1$. 

When $p_4$ approaches $p_3$, since $\gamma_3\gamma_4 = id$ consists of $d$ cycles of length 1, the degenerate covering curve has $d$ nodes. Its two components $C_{12}$ and $C_{34}$ are both rational and they meet at $d$ nodes whose local branches map to $p_0$ as $(x,y)\rightarrow (x, y)$. 
There is no unstable rational component, hence each node contributes $1$ and
$\delta_3 ({\bf r}) = \delta_3' ({\bf r}) = d$. The same holds for the case when $p_4$ approaches $p_1$, so $\delta_1 ({\bf r}) = \delta_1' ({\bf r}) = d$. 
When $p_4$ approaches $p_2$, since $d$ is odd, $\gamma_2\gamma_4 = \gamma^2$ consists of a single cycle of length $d$. The degenerate covering curve has a single node 
whose local branches map to $p_0$ like $(x,y)\rightarrow (x^d, y^d)$. Its two components $C_{12}$ and $C_{34}$ both have genus $\frac{d-1}{2}$ and there is no unstable rational component. 
Hence, the unique node contributes $\frac{1}{d}$ and $\delta_2 ({\bf r}) = \delta_2' ({\bf r}) = \frac{1}{d}$.

Since $|\OO_1| = 1$, we have $\delta_{\OO_1} = \delta'_{\OO_1} = d + d + \frac{1}{d} = \frac{2d^2+1}{d}.$ By the slope formula in Theorem~\ref{slope}, 
we obtain $$s(\overline{Z}_{\OO_1}) = \frac{8d^2+4}{d^2-1} = 8+\frac{12}{g^2+2g}.$$ 

Further assume that $d$ is coprime with 6. There is another orbit $\OO_2$ of $Cov_d({\bf c})$ that consists of a unique 
monodromy data ${\bf r} = (\gamma, \gamma, \gamma, \gamma^{-3})$ (up to equivalence) under the actions generated by $g_1, g_2$ in Theorem~\ref{monodromy}, where 
$\gamma$ is any cycle of length $d$ in $S_d$. Hence, we have the cardinality $|\OO_2| = 1$. The genus $g$ of the covers equals $d-1$. By a similar analysis as above, we get 
$\delta_1 ({\bf r}) = \delta_2({\bf r}) = \delta_3 ({\bf r}) = \frac{1}{d}$. Since $|\OO_2| = 1$, we have $\delta_{\OO_2} = \delta'_{\OO_2} = \frac{3}{d}.$ By the slope formula, we obtain $$s(\overline{Z}_{\OO_2}) = \frac{36}{d^2-1} = \frac{36}{g^2+2g}.$$

Note that $s(\overline{Z}_{\OO_1})$ converges to 8 while $s(\overline{Z}_{\OO_2})$ converges to 0 for $g\gg 0$. 
\end{example}

The covers in Example~\ref{application} are special cases of cyclic covers of $\PP^1$. In the next section, we will come up with a general formula to calculate the slope of Hurwitz spaces of cyclic covers. 

Let us also consider an example of non-cyclic covers. 

\begin{example}
\label{non-cyclic}
Consider $d = 4$ and ${\bf r} = ((1234), (1432), (123), (132))$. The notation $(a_1\ldots a_k)$ stands for a permutation in $S_d$ that sends $a_i$ to $a_{i+1}$. The orbit $\OO$ generated by ${\bf r}$ under the action in Theorem~\ref{monodromy}  
consists of six elements (up to equivalence): 
$$ {\bf r_1} = ((1234), (1432), (123), (132)); \ {\bf r_2} = ((1234), (1432), (132), (123)); $$
$$ {\bf r_3} = ((1234), (1324), (123), (243)); \ {\bf r_4} = ((1234), (1324), (134), (123)); $$
$$ {\bf r_5} = ((1234), (1324), (243), (134)); \ {\bf r_6} = ((1234), (1234), (123), (124)). $$
By Riemann-Hurwitz, the genus of these covers equals 2. Hence, we get an irreducible family $\overline{Z}_{\OO}$  
of degree 4, genus 2 covers of $\PP^1$ with two totally ramified points and two triply ramified points. Since the cardinality $|\OO|$ equals 6, 
$\overline{Z}_{\OO} \rightarrow \Mof$ is a finite morphism of degree 6. 

When $p_4$ meets $p_3$, the degenerate admissible cover arising nearby ${\bf r_1}$ looks as follows: 
\begin{figure}[H]
    \centering
    \psfrag{P}{$\PP^1$}
    \psfrag{2}{$p_{2}$}
    \psfrag{3}{$p_{3}$}
    \psfrag{4}{$p_{4}$}
    \psfrag{1}{$p_{1}$}
      \includegraphics[scale=0.5]{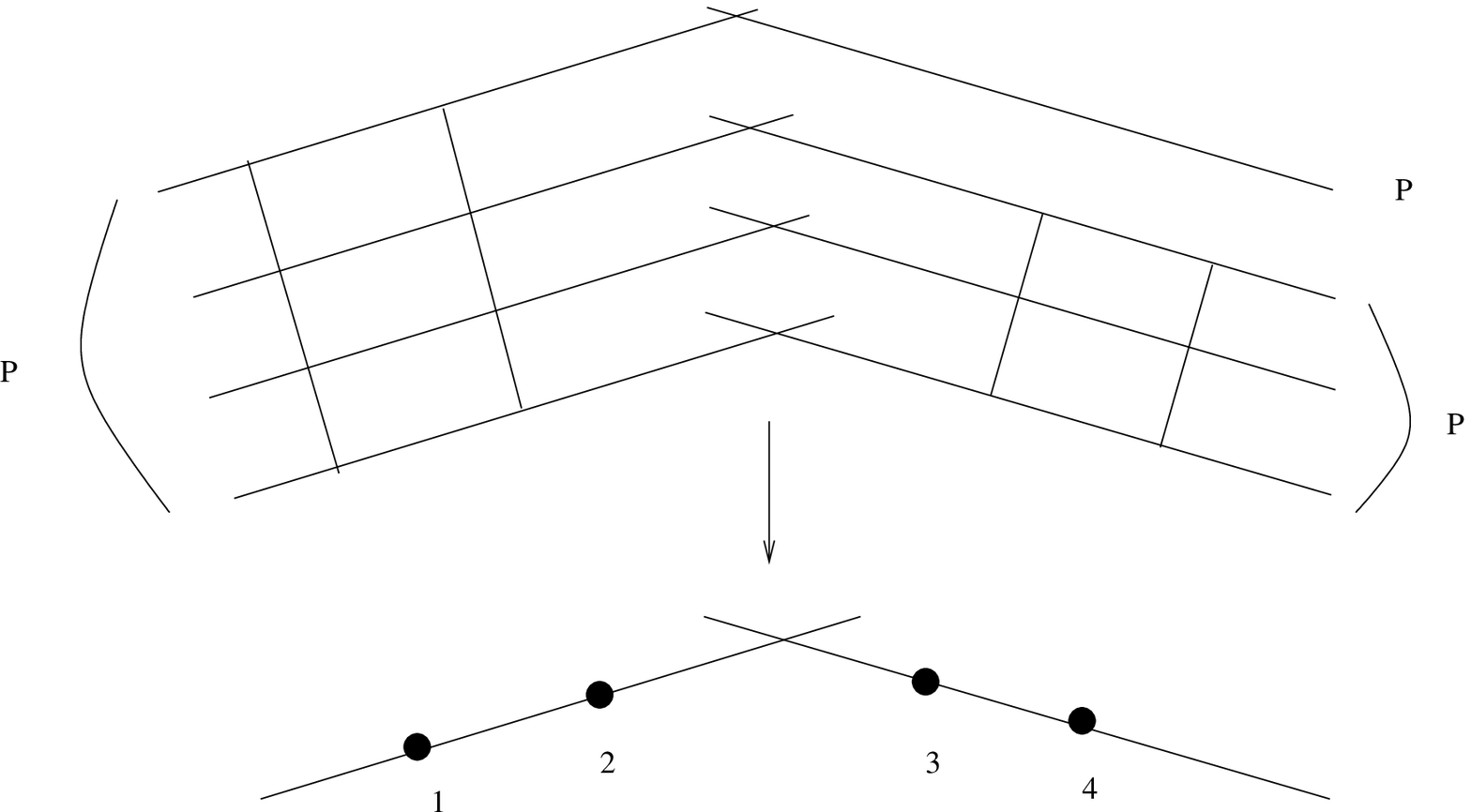}
    \end{figure}
The covering curve consists of a $\PP^1$ meeting another $\PP^1$ at three nodes and a third $\PP^1$ at one node. Its stabilization consists of the first two $\PP^1$ components meeting at three nodes. Around any node of the cover, the map is a local isomorphism, so each node has weight $\frac{1}{1} = 1$. The four nodes in total contribute 4 to $\delta'$ by definition. The node on top of the figure belongs to the semistable $\PP^1$ component that maps to a smooth point of the stabilization, so it does not contribute to $\delta$ by definition. The other three nodes in total contribute 3 to $\delta$. In the same way, one can analyze the degenerate admissible cover nearby ${\bf r_i}$ when $p_4$ meets $p_j$ and count the contributions to $\delta'$ and $\delta$, for $i = 1, \ldots, 6$ and $j = 1,2,3$. We skip the enumeration and just write down the result: $\delta'_{\OO} = 25$ and $\delta_{\OO} = 23$. By the slope formula, we obtain that 
$$s(\overline{Z}_{\OO}) =  \frac{12\cdot 23}{25 + 6\Big({\displaystyle 4 - \frac{1}{4} - \frac{1}{4} - \big(1+\frac{1}{3}\big) - \big(1 + \frac{1}{3}}\big)\Big)} = 9\frac{1}{5}. $$  
\end{example}

\section{Cyclic covers}
In this section, we fix the monodromy profile ${\bf r} = (\gamma^{a_1}, \gamma^{a_2}, \gamma^{a_3}, \gamma^{a_4})$, where $\gamma$ is a cycle of length $d$ in $S_d$, 
$\sum_{i=1}^{4} a_i  \equiv 0$ (mod $d$), $\mbox{gcd} (a_1, a_2, a_3, a_4, d) = 1$, and $1\leq a_i \leq d-1$. Under the actions $g_1, g_2$ in Theorem~\ref{monodromy}, it is easy to see that all the monodromy data in the orbit $\OO$ generated by ${\bf r}$ are equivalent up to the $S_d$ conjugate actions. Hence, the morphism $\ZZ\rightarrow \Mof$ is one to one.  

We can concretely write down such a cover corresponding to ${\bf r} = (\gamma^{a_1}, \gamma^{a_2}, \gamma^{a_3}, \gamma^{a_4})$ as
$$ y^{d} = (x-z_1)^{a_1} (x-z_2)^{a_2} (x-z_3)^{a_3} (x-z_4)^{a_4}, $$
where $z_1, z_2, z_3, z_4$ are four distinct points in $\PP^1$. 
The covering map is induced by $(x, y) \rightarrow x$. 
Given $\sum_{i=1}^{4} a_i  \equiv 0$ (mod $d$) and $\mbox{gcd} (a_1, a_2, a_3, a_4, d) = 1$, it is a cyclic cover with 
Galois group $\mathbb Z / d$ branched at four points $z_1, z_2, z_3, z_4$, cf. e.g., \cite{R}. Since $\mbox{gcd} (a_1, a_2, a_3, a_4, d) = 1$, the group $\langle \gamma^{a_1}, \gamma^{a_2}, \gamma^{a_3}, \gamma^{a_4} \rangle $ is the same as $\langle \gamma\rangle $, which acts transitively on the $d$ letters. So the covering curve is connected. 

Now let $z_4$ be the moving point in $\PP^1$. In order to apply Theorem~\ref{slope}, we need to analyze the degenerate covers when $z_4$ meets $z_1, z_2$ or $z_3$. 
Set up some notation as follows: 
$$d_i = \mbox{gcd} (a_i, d), \ 1\leq i \leq 4,$$
$$s_{ij} = \mbox{gcd} (a_i+a_j, d), \ 1\leq i < j\leq 4,$$ 
$$t_{ij} = \mbox{gcd} (a_i, a_j, d), \ 1\leq i < j \leq 4. $$
Note that by the assumption $\sum_{i=1}^{4} a_i  \equiv 0$ (mod $d$), we have $s_{ij} = s_{kl}$ for $\{i,j,k,l \} = \{  1,2,3,4 \}$. 

Let us study the degenerate cover when $z_4$ meets $z_3$. The other two cases follow by changing subindices accordingly. 
Let $\pi_{0}: C_0\rightarrow Q_{12} \cup_{p_0} Q_{34}$ be the degenerate admissible cover when $z_4$ meets $z_3$, where 
$  Q_{12} \cup_{p_0} Q_{34} $ is the nodal union of two $\PP^1$'s at the node $p_0$. As we have seen in the introduction, $C_0$ consists of two components $C_{12}$ 
and $C_{34}$ such that $\pi_{0}$ restricted to $C_{12}$ and $C_{34}$ are degree $d$ covers of $Q_{12}$ and $Q_{34}$, respectively. Furthermore, 
$C_{12} \rightarrow Q_{12}$ is branched at $z_1, z_2, p_0$ with the monodromy data $(\gamma^{a_1}, \gamma^{a_2}, \gamma^{-(a_1 + a_2)})$ and 
$C_{34} \rightarrow Q_{34}$ is branched at $z_3, z_4, p_0$ with the monodromy data $(\gamma^{a_3}, \gamma^{a_4}, \gamma^{-(a_3 + a_4)})$. 
In particular, $\gamma^{-(a_1 + a_2)} $ consists of $s_{12} $ cycles of length $d/s_{12}$. Therefore, $C_{12}$ and $C_{34}$ intersect at 
$s_{12}$ nodes. Locally around each node, $\pi_0$ is given by $(x, y )\rightarrow (u, v) = (x^{d/s_{12}}, y^{d/s_{12}})$. By definition, we have 
$$\delta'_3({\bf r}) = s_{12}\frac{1}{d/s_{12}} = \frac{s_{12}^2}{d}. $$ 

For the value of $\delta_3({\bf r})$, it depends on whether or not a node belongs to a component of $C_0$ that can be contracted to a smooth point of the stabilization $C_0^{st}$. 
Denote by $R$ a rational tail of $C_0$ if $R$ is a smooth rational component that meets $\overline{C_0\backslash R}$ at one point. 

\begin{lemma}
\label{tail}
The covering curve $C_0$ does not have any rational tail. 
\end{lemma}

\begin{proof}
Note that the group $\langle  \gamma^{a_1}, \gamma^{a_2}, \gamma^{-(a_1 + a_2)} \rangle $ equals $\langle  \gamma^{t_{12}}\rangle$. 
Its action on the $d$ letters has $t_{12}$ orbits and each orbit contains $d/t_{12}$ letters. It implies that $C_{12}$ has $t_{12}$ connected components
and $\pi_0$ restricted to each component is a degree $d/t_{12}$ cover of $Q_{12}$ with the monodromy data
$(\eta^{a_1/t_{12}}, \eta^{a_2/t_{12}}, \eta^{-(a_1+a_2)/t_{12}})$ over the three branch points $z_1, z_2, p_0$, where $\eta$ is a single cycle of length $d/t_{12}$. 

By $\mbox{gcd}(a_1/t_{12}, a_2/t_{12}, d/t_{12}) = 1$, the monodromy data $(\eta^{a_1/t_{12}}, \eta^{a_2/t_{12}}, \eta^{-(a_1+a_2)/t_{12}})$ yield a connected 
cover $R$ of degree $d/t_{12}$ to $Q_{12}\cong \PP^1$. If $R$ is a rational tail of $C_0$, then $R$ meets $C_{34}$ at a unique node. It implies that $ \eta^{-(a_1+a_2)/t_{12}} $ consists of a single cycle of length 
$d/t_{12}$, namely, $(a_1 + a_2 )/t_{12}$ and $d/t_{12}$ are coprime. Moreover, $\mbox{gcd}(a_i/t_{12}, d/t_{12}) = d_i/t_{12}$ for $i = 1, 2$, hence 
$\eta^{a_i/t_{12}}$ consists of $d_i/t_{12}$ cycles of length $d/d_i$. By the Riemann-Hurwitz formula, the genus of $R$ satisfies 
$$2g_R - 2 + 2d/t_{12} = (d/t_{12} - 1) +  (d/t_{12} - d_1/t_{12}) + (d/t_{12} - d_2/t_{12}),  $$
$$ 2g_R = \frac{d-d_1-d_2}{t_{12}} + 1. $$
Since $g_R = 0$, we get $d < d_1 + d_2$. But $d_i = \mbox{gcd} (a_i, d) \leq d/2$ for $i = 1,2$, contradiction. 
\end{proof}

Since $C_0$ does not have any rational tail, its nodes map to the nodes of $C_0^{st}$. By definition, we have 
$$\delta_3({\bf r}) = \delta'_3({\bf r}) =\frac{s_{12}^2}{d}. $$ 

\begin{theorem}
\label{cyclic}
The Hurwitz component $\ZZ$ of cyclic covers of $\PP^1$ with four branch points has slope 
$$ s(\ZZ) = \frac{12(s^2_{12}+s^2_{23}+s^2_{13})}{s^2_{12}+s^2_{23}+s^2_{13} + d^2 - \sum\limits_{i=1}^4 d_i^2}, $$
where $s_{ij} = gcd(a_i+a_j, d)$ and $d_i = gcd(a_i,d).$ 
\end{theorem}

\begin{proof}
We have seen that $\delta_3({\bf r}) = \delta'_3({\bf r}) =s_{12}^2/d$. Similarly, we have $\delta_1({\bf r}) = \delta'_1({\bf r}) =s_{23}^2/d$
and $\delta_2({\bf r}) = \delta'_2({\bf r}) =s_{13}^2/d$. Since the orbit $\OO$ generated by ${\bf r}$ under the actions in Theorem~\ref{monodromy} 
contains a unique equivalent class up to the $S_d$ conjugate actions, by definition, we have $|\OO| = 1$ and  
$$\delta_{\OO} = \delta'_{\OO} =  \frac{s^2_{12}+s^2_{23}+s^2_{13}}{d}. $$ 
Moreover, $\gamma^{a_i}$ consists of $d_i$ cycles of length $d/d_i$ for $1\leq i\leq 4$. 
By Theorem~\ref{slope}, we have 
$$s(\ZZ) = \frac{\frac{12}{d}(s^2_{12}+s^2_{23}+s^2_{13})}{\frac{1}{d}(s^2_{12}+s^2_{23}+s^2_{13}) + d - \sum\limits_{i=1}^4 \frac{d^2_i}{d}}. $$
After simplifying, we thus obtain the desired expression for $s(\ZZ)$. 
\end{proof}

Let us revisit the two special cyclic covers $\OO_1$ and $\OO_2$ in Example~\ref{application}. For $\OO_1$, it corresponds to the case when 
$a_1 = a_3 = 1, a_2 = a_4 = d-1$ and $d$ is odd. Then we have $d_i = 1$ for $1\leq i\leq 4$, $s_{12} = s_{23} = d$ and $s_{13} = 1$. 
By Theorem~\ref{cyclic}, we have 
$$ s(\overline{Z}_{\OO_1}) = \frac{8d^2+4}{d^2-1}.$$ 
For $\OO_2$, it corresponds to the case when $a_1 = a_2 = a_3 = 1, a_4 = d-3$ and $d$ is coprime with 6. Then we have 
$d_i = 1$ for $1\leq i \leq 4$ and $s_{ij} = 1$ for $1\leq i < j \leq 3$. 
By Theorem~\ref{cyclic}, we have 
$$ s(\overline{Z}_{\OO_2}) = \frac{36}{d^2-1}. $$ 
These results coincide with our previous analysis in Example~\ref{application}.  

\begin{remark}
In \cite{C1}, the author considered covers of an elliptic curve with a unique branch point. For fixed $g$, there exist such covers of arbitrarily large degree and the limit of slopes of corresponding Hurwitz spaces can provide a lower bound for the slope of effective divisors on $\Mg$. Nevertheless, the degree is bounded for genus $g$ cyclic covers of $\PP^1$ with four branch points. In fact, by the Riemann-Hurwitz formula, we know $$ g = d +1 - \frac{1}{2} \sum_{i=1}^4  \mbox{gcd} (a_i, d) $$ 
for a cyclic cover given by $y^{d} = (x-z_1)^{a_1} (x-z_2)^{a_2} (x-z_3)^{a_3} (x-z_4)^{a_4}$. 
Since $1\leq a_{i}\leq d-1$ for $1\leq i\leq 4$ and $\mbox{gcd} (a_1, a_2, a_3, a_4, d) = 1$, we have 
$$\sum_{i=1}^4 \mbox{gcd} (a_i, d) \leq \frac{d}{2}+ \frac{d}{2} + \frac{d}{2} +\frac{d}{3} = \frac{11d}{6}. $$
Hence, $ g \geq d+1 - \frac{11d}{12}$ and $d \leq 12(g-1)$. 
\end{remark}

For an effective divisor $D$ in $\Mg$ with divisor class $a\lambda -\sum_{i=0}^{[g/2]}b_i\delta_i$, define its slope 
$$s(D) = \frac{a}{\mbox{min}\{b_i\}},$$ 
where $a, b_i > 0$. If an irreducible curve $C$ in $\Mg$ has slope $s(C) > s(D)$, then 
it implies $C\ldotp D < 0$ and $C$ is contained in $D$. By this argument, our calculation of $s(\ZZ)$ can induce various applications 
regarding the geometry of covering curves in the Hurwitz space $\ZZ$. For instance, the divisor $\Theta$ on $\Mg$ parameterizing curves with an even theta-characteristic 
$\eta$ such that $h^{0}(\eta) > 0$ has slope $8 + 1/2^{g-3}$. The Hurwitz space $\overline{Z}_{\OO_1}$ in Example~\ref{application} has slope 
$$s(\overline{Z}_{\OO_1}) =  \frac{8d^2+4}{d^2-1} = 8+\frac{12}{g^2+2g}$$ 
for odd $d$ and even $g= d-1$. When $g\geq 6$, we have $s(\overline{Z}_{\OO_1}) > s(\Theta)$, hence a curve in $\overline{Z}_{\OO_1}$ can admit an even theta-characteristic $\eta$ with $h^{0}(\eta) > 0$. 

Consider a cyclic cover 
$\pi: C \rightarrow \PP^1$ defined by 
$$ y^{d} = (x-z_1)^{a_1} (x-z_2)^{a_2} (x-z_3)^{a_3} (x-z_4)^{a_4}.$$
The quadratic differential 
$$ q_0 = \frac{(dz)^2}{(z-z_1)(z-z_2)(z-z_3)(z-z_4)} $$
defines a flat metric on $\PP^1$. For a suitable choice of $z_1,z_2, z_3, z_4$, the sphere can be seen as a pillow case by gluing 
two copies of a unit square along their boundary. Consequently $C$ can be tiled by $2d$ unit squares, i.e. it is a square-tiled surface, 
cf. \cite{FMZ} for a detailed description. 

Regard the pair $(C, \pi^{*}(dz))$ as a point in the moduli space of quadratic differentials 
$\mathcal Q$. There is an induced SL($2,\RR$) action on $\mathcal Q$ by acting on the real and imaginary parts of a quadratic differential. 
Note that the orbit generated by $(C, \pi^{*}(dz))$ is the component $Z_{\OO}$ of the 1-dimensional Hurwitz space containing the cover $\pi$ (up to a finite base change, depending on whether the branch points are ordered), since the SL($2,\RR$) action on $(C, \pi^{*}(dz))$ amounts to varying the 
cross-ratio of $z_1,z_2, z_3, z_4$. Thus $\ZZ$ projects to $\Mg$ as a closed algebraic curve, which is a Teichm\"{u}ller curve. In general, it is well-known that the SL($2,\RR$) orbit of a square-tiled surface yields a so-called arithmetic Teichm\"{u}ller curve. 

Let $L$ denote the sum of the non-negative Lyapunov exponents of the Hodge bundle over a Teichm\"{u}ller curve. Roughly speaking, the Laypunov exponents measure the growth rate of the length of a vector in the bundle under parallel transport along the Teichm\"{u}ller geodesic flow, cf. \cite{K} for an introduction on Lyapunov exponents. For a Teichm\"{u}ller curve $\ZZ$ parameterizing cyclic covers of $\PP^1$ with four branch points, the sum $L$ was calculated in \cite{FMZ}. In fact, our analysis for the slope of $\ZZ$ can also deduce the value of $L$. 

\begin{theorem}
\label{Lcyclic}
The sum of Lyapunov exponents of the Hodge bundle over $\ZZ$ is equal to 
$$ L =  \frac{d}{6}- \frac{1}{6d}\sum_{i=1}^4 gcd^2(d, a_i) + \frac{1}{6d}\Big(gcd^2(d, a_1+a_2) + gcd^2(d, a_1+a_3)+ gcd^2(d, a_2+a_3)\Big).$$ 
\end{theorem}

\begin{proof}
By a slight variation of the Kontsevich formula \cite{K}, cf. \cite{BM} and \cite{EKZ2}, we have 
$$ L = \frac{2(\mbox{deg}\ \lambda|_{\ZZ})}{2g(\ZZ) - 2 + s}, $$
where $g(\ZZ)$ is the genus of $\ZZ$ and $s$ is the number of singular covers. The Hurwitz component 
$\ZZ$ is isomorphic to $\PP^1$ and it contains three singular covers, when the moving branch point $z_4$ meets one of the 
fixed branch points $z_1, z_2, z_3$. Therefore, we get $ L  =  2(\mbox{deg}\ \lambda|_{\ZZ}). $
The degree of $\lambda$ has been calculated in the proof of Theorem~\ref{slope} in a more general setting, 
and worked out explicitly in the proof of Theorem~\ref{cyclic} for cyclic covers. Using these data, we thus obtain 
the desired equality for $L$. 
\end{proof}

\begin{remark}
Individual Laypunov exponents of $\ZZ$ were further calculated in \cite{EKZ2}. For a general Teichm\"{u}ller curve, there is an analogue of the Kontsevich formula to interpret individual Laypunov exponents by the degree of local systems \cite{BM}.  
\end{remark}

\section{Quadratic differentials}
We first introduce the moduli space of quadratic differentials, cf. e.g., \cite{EO2} and \cite{L} for more details. 

Let $\mu= (2m_1, \ldots, 2m_k)$ and $\nu = (2n_1-1, \ldots, 2n_l-1)$ be two partitions, where $m_i, n_j \geq 1$ and $l$ is even. Let $\QQ (\mu, \nu)$ denote the moduli space of quadratic differentials parameterizing pairs $(C, \psi)$, where $\psi$ is a quadratic differential of a smooth curve $C$, the divisor $(\psi) = \sum 2m_ip_i + \sum (2n_j - 1) q_j$ and $\psi$ is not a global square of an Abelian differential. The genus $g$ of $C$ satisfies $$ 4g-4 = \sum 2m_i + \sum (2n_j-1). $$ 

One can associate $(C, \psi)$ a canonical double cover $\pi: \tilde{C} \rightarrow C$ such that $\pi^{*}\psi = \omega^2$, where $\omega$ is an Abelian differential on $\tilde C$, cf. \cite[Construction 1]{L}. 
The double cover $\pi$ is branched exactly at $q_1, \ldots, q_l$ the zeros of $\psi$ with odd multiplicities. Let $\pi^{-1}(p_i) = \{p_{i}', p_{i}''\}$ and $\pi^{-1}(q_{j}) = q_j'$. One checks that 
$(\omega) = \sum m_i (p_i' + p_i'') + \sum 2n_j q_j'$. Namely, $(\tilde C, \omega)$ is parameterized in $\HH (m_i, m_i, 2n_j)$ the corresponding moduli space of Abelian differentials. The genus $\tilde g$ of $\tilde C$ satisfies 
$$ 2\tilde g-2 = 2(2g-2) + l. $$

Let $\sigma: \tilde C\rightarrow \tilde C$ be the involution such that $C \cong \tilde C/ \sigma$. Since $\psi$ is not a global square, we have $\sigma^{*}\omega = - \omega$. The relative homology group $H_1 (\tilde C, p_i', p_i'', q_j'; \mathbb C)$ splits into a direct sum of two eigenspaces $H_{+}\oplus H_{-}$ under the involution $\sigma$, where $H_{+} = \pi^{*}H_{1}(C, p_i, q_j ; \mathbb C)$. For any $\gamma\in H_{+}$, we have $\int_{\gamma} \omega = 0$. Then $H_{+}$ has dimension 
$$n = (2\tilde{g} + 2k + l - 1) - (2g + k + l -1) = 2g -2 + k + l.$$
Let $\gamma_1, \ldots, \gamma_n$ be a basis of $H_{-}$. It was noticed by Kontsevich that 
the period map 
$$\Phi: (C, \psi) \rightarrow \Big(\int_{\gamma_1}\omega, \ldots, \int_{\gamma_n} \omega\Big) \in \mathbb C^n $$ 
yields a local coordinate system for $\QQ (\mu, \nu)$. 

Let $\mathbb T = \mathbb C/\mathbb Z^2$ be the standard torus. Consider the double cover $\rho: \mathbb T\rightarrow \PP^1 \cong \mathbb T/\pm$. In this way, $\PP^1$ can be regarded as a pillow case with four branch points at 
$$z_1 = 0, \ z_2 = \frac{1}{2}, \ z_3 = \frac{\sqrt{-1}}{2}, \ z_4 = \frac{1+\sqrt{-1}}{2}.$$
The quadratic differential $(dz)^2$ on $\mathbb T$ descends to $\PP^1$ and it has a simple pole at each $z_i$. Define the covering set $$ Cov(\mu, \nu) = \{f: C\rightarrow \PP^1\}$$ parameterizing connected covers of even degree such that $f$ is branched at $z_1$ of ramification type $(2n_1+1, \ldots, 2n_l+1, 2, \ldots, 2)$, branched at $z_i$ for $2\leq i\leq 4$ of type $(2, \ldots, 2)$ and branched at $k$ other fixed points of type $(m_i+1, 1, \ldots, 1)$ for $1\leq i \leq k$. For any $f\in Cov(\mu, \nu)$, one checks that $(C, f^{*}(dz)^2) \in \QQ (\mu, \nu)$. Moreover, by the construction of canonical double covers, there exists a unique cover
$\tilde f: \tilde C\rightarrow \mathbb T$ that completes the following commutative
diagram: 
$$\xymatrix{
\tilde C \ar[r]^{\pi} \ar[d]_{\tilde f}  & C \ar[d]^{f} \\
\mathbb T \ar[r]^{\rho}     & \PP^1} $$ 

We take an explicit basis $\gamma_1, \ldots, \gamma_n$ of $H_{-}$ as follows. Let $\gamma_1,\ldots, \gamma_{2g-2+l}$ be a basis of $H_1(\tilde C; \mathbb Z)/\pi^{*}H_1(C; \mathbb Z)$. Let $\gamma_{2g-2+ l + i} = \overline{p_i'q_1'} - \sigma (\overline{p_i'q_1'})$ for $1\leq i \leq k$, where $\overline{pq}$ means a path connecting $p$ and $q$. See the picture below. 

\begin{figure}[H]
    \centering
    \psfrag{1}{$p'_{1}$}
    \psfrag{2}{$p''_{1}$}
    \psfrag{3}{$p'_{k}$}
    \psfrag{4}{$p''_{k}$}
    \psfrag{q}{$q'_1$}
    \includegraphics[scale=0.5]{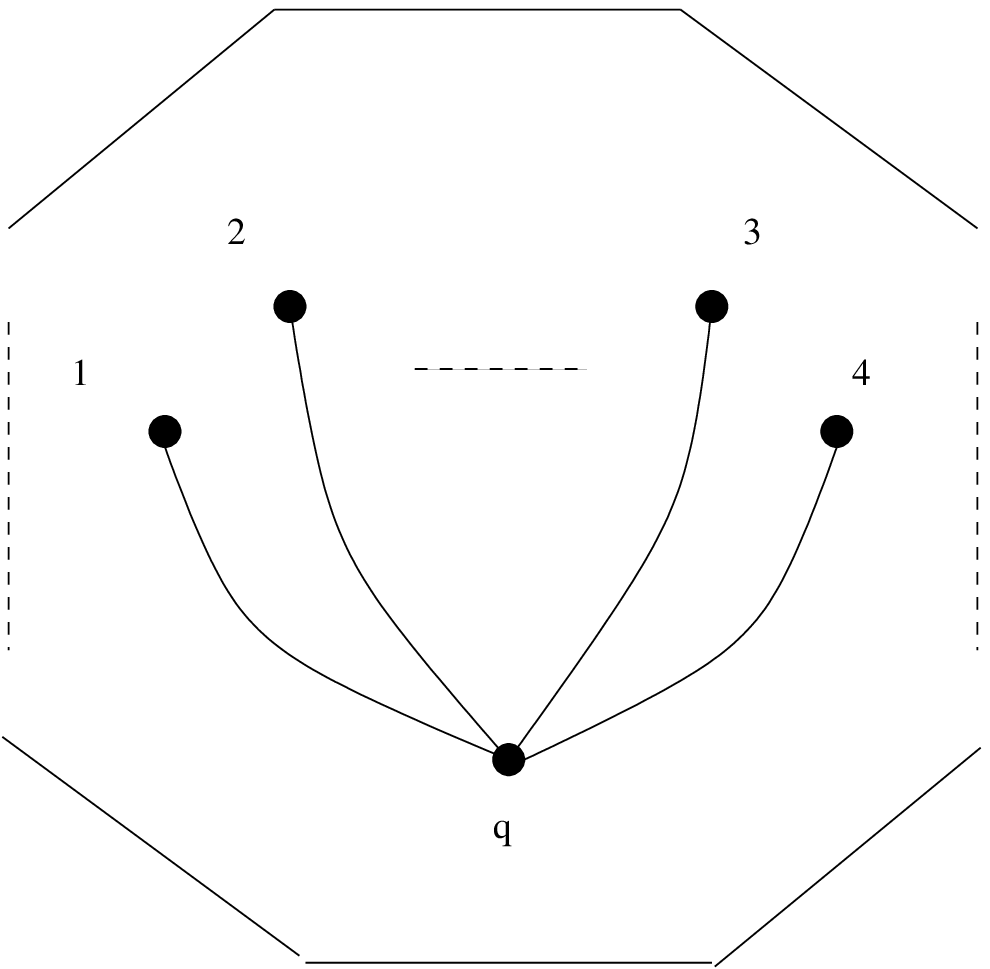}
    \end{figure}

The following result implies that covers in $Cov(\mu, \nu)$ correspond to a shift of lattice points in $\QQ (\mu, \nu)$ under the period map. It was stated and used in \cite{EO2} for evaluating the volume of 
$\QQ (\mu, \nu)$. Its analogue for Abelian differentials was proved in \cite[Lemma 3.1]{EO1}. 

\begin{lemma}
\label{EOQ}
For $(C, \psi)\in \QQ (\mu, \nu)$, consider its coordinates 
$\Phi(C, \psi) = (\phi_1, \ldots, \phi_n) \in \mathbb C^n$, where $\phi_{i_1}\neq\phi_{i_2}, z_j$ mod $(\mathbb Z^2, \pm)$ for 
any $i_1, i_2 > 2g-2 + l$ and $1\leq j \leq 4$. Then we have $\phi_i \in \mathbb Z^2$ for $1\leq i \leq 2g-2+l$ if and only if 
 the following conditions hold: 
 
\noindent (1) there exists a cover $f: C \rightarrow \PP^1 = \mathbb T/\pm$ and $\psi = f^{*}(dz)^2$; 
 
 \noindent (2) $f(q_j) = z_1 = 0$ for $1\leq j\leq l$ and $f(p_i)  = \phi_{2g-2+l+i}$ (mod $\mathbb Z^2, \pm$) for $1\leq i\leq k$;
 
 \noindent (3) the ramification of $p_i$ is of the form $z \rightarrow z^{m_i+1}$ and the ramification of $q_j$ is
 $z\rightarrow z^{2n_j+1}$; 
 
\noindent (4) $f$ is branched at $z_1$ of ramification type $(2n_1+1, \ldots, 2n_l+1, 2, \ldots, 2)$, branched at $z_j$ of type $(2, \ldots, 2)$ for $2\leq j\leq 4$ and branched at $\phi_{2g-2+l+i}$ of type $(m_i+1, 1, \ldots, 1)$ for $1\leq i \leq k$.  
\end{lemma}

\begin{proof}
If such a map $f$ exists, by the above diagram, we know $\pi^{*}\psi = \omega^2$ and $\omega = \tilde{f}^{*}dz$. Then we have 
$$\int_{\gamma_i}\omega = \int_{\gamma_i} \tilde f^{*}dz = \int_{\tilde f_{*}\gamma_i}dz \in \mathbb Z^2$$  
for $1\leq i\leq 2g-2+l$, since $\tilde f_{*}\gamma_i$'s are closed paths on $\mathbb T$. 
Moreover, we have 
$$\int_{\gamma_{2g-2+l+i}} \omega = \int_{\tilde{f}_{*}(\overline{p_i'q_1'})}dz -  \int_{\tilde{f}_{*}\sigma(\overline{p_i'q_1'})}dz
= \int_{\tilde f_{*}(\overline{p_i'p_i''})}dz \in \mathbb Z^2 $$ 
for $1\leq i\leq k$, since $\tilde{f}$ maps $p_i'$ and $p_i''$ to the same point in $\mathbb T$. 

For the other direction, define $$f(z) = \int_{z'}^{\sigma (z')}\omega,$$ 
where $z\in C$ and $z' \in \pi^{-1}(z)$. 
Since $\phi_i \in \mathbb Z^2$ for $1\leq i\leq 2g-2 + l$, modulo $\mathbb Z^2$, $f(z)$ does not depend on the integration path. Modulo $\pm$, $f(z)$ is independent of the choice of $z' \in \pi^{-1}(z)$. Therefore, we obtain an induced map $f: C \rightarrow \PP^1 \cong \mathbb T/\pm$. It is straightforward to check that $f$ satisfies all the desired conditions. 
\end{proof}

Up to a scalar, $\QQ (\mu, \nu)/\mathbb C^{*}$ has dimension equal to $2g - 3 + k + l$ and it naturally maps to $\mathcal M_g$. 
When dim $\QQ (\mu, \nu)/\mathbb C^{*}\geq$ dim $\mathcal M_g$, we prove that $\QQ (\mu, \nu)/\mathbb C^{*}\rightarrow \mathcal M_g$ is dominant.  

\begin{lemma}
\label{dominant}
For $k + l \geq g$, a general genus $g$ curve admits a quadratic differential contained in $\QQ (\mu, \nu)$. 
\end{lemma}

\begin{proof}
There is a stratification among all $\QQ(\mu, \nu)$'s, so it suffices to verify the boundary case when $k + l = g$. We apply the 
De Jonqui\`{e}res' Formula \cite[VIII \S 5]{ACGH}. Let $a_{1}, \ldots, a_{m}$ be distinct integers such that $a_{i}$ appears $h_{i}$ times in the partition $(\mu, \nu)$ of $4g-4$. We have $\sum_{i=1}^{m}h_{i} = k+l = g$ and $\sum_{i=1}^{m}h_{i}a_{i} = 4g-4$. 
Define $R(t) = 1+\sum_{i=1}^{m}a_{i}^{2}t_{i}$. On a general genus $g$ curve, the virtual number of quadratic differentials (up to a scalar) that have $h_i$ zeros of multiplicity $a_i$ is 
$$ \left[R(t)^{g}\right]_{t_{1}^{h_{1}}\cdots t_{m}^{h_{m}}}. $$ 
Note that $R(t)$ has positive coefficients, hence the virtual number is also positive, which implies the existence of such 
quadratic differentials (possibly infinitely many if the curve is special).   
\end{proof}

\begin{corollary}
\label{dense}
For $k + l \geq g$, the union of genus $g$ curves that admit covers in $Cov(\mu, \nu)$ forms a Zariski dense subset of
$\mathcal M_g$.  
\end{corollary}

\begin{proof}
Since covers in $Cov(\mu, \nu)$ correspond to certain shift of lattice points under the period map, their union is Zariski dense in $\QQ (\mu, \nu)$. Now the conclusion follows from Lemma~\ref{dominant}. 
\end{proof}

Let $\mu$ be null and $\nu = (2n_1-1,\ldots, 2n_l-1)$ a partition of $4g-4$. Let $c_1$ be the conjugacy class $(2n_1 +1,\ldots, 2n_l +1, 
2,\ldots, 2)$ and $c_2, c_3, c_4$ the same conjugacy class $(2, \ldots, 2)$ of $S_d$, where $d$ is even. Let ${\bf c_{\nu}} = (c_1, c_2, c_3, c_4)$ 
denote the ramification profile. By the notation in the introduction, the set
$Cov_d({\bf c_{\nu}})$ parameterizes degree $d$, genus $g$ connected covers of $\PP^1$ with four fixed branch points and the ramification profile ${\bf c_{\nu}}$. 
Let $\OO \subset Cov_d({\bf c_{\nu}})$ denote an orbit of the action in Theorem~\ref{monodromy} and $Z_{\OO}$ the corresponding irreducible component of the Hurwitz space $\HH_d({\bf c_{\nu}})$. 

\begin{corollary}
\label{bounding}
For $ l \geq g$, the union of irreducible components of $\HH_d({\bf c_{\nu}})$ for all $d$ maps to a Zariski dense subset 
in $\mathcal M_g$. Let $s_{\nu}$ be the limit (if exists) of their slopes as $d$ approaches infinity. 
For an effective divisor $D$ on $\Mg$, it has slope $s(D) \geq s_{\nu}$. 
\end{corollary}

\begin{proof}
For $l\geq g$, the density result follows from Corollary~\ref{dense}. If $D$ is an effective divisor on $\Mg$, there exist infinitely many Hurwitz components $Z_{\OO}$ not entirely contained in $D$. We have the intersection number $D\ldotp \ZZ \geq 0$, hence $s(D) \geq s(\ZZ)$. Taking the limit when $d$ goes to infinity, 
we get $s(D) \geq s_{\nu}$. 
\end{proof}

In order to evaluate $s_{\nu}$, we first simplify the slope formula for components of $\HH_d({\bf c_{\nu}})$. 

\begin{lemma}
\label{notail}
Let $\pi_{0}: C_0 \rightarrow Q_{12}\cup_{z_0}Q_{34}$ be a degenerate cover parameterized in $\overline{\HH}_d({\bf c_{\nu}})$, where $Q_{12}, Q_{34}$ are two $\PP^1$'s containing the branch points $z_1, z_2$ and $z_3, z_4$, respectively, and they meet at the node $z_0$. Then $C_0$ does not have a rational tail. 
\end{lemma}

\begin{proof}
We prove by contradiction. If $C'$ is a rational tail of $C_0$, let $z' = C' \cap \overline{C_0\backslash C'}$ and $d'$ the degree of $\pi_0$ restricted to $C'$. Then $z'$ is the unique pre-image of the node $z_0$ in $C'$. 

If $C'$ maps to $Q_{34}$, it is branched at $z_1, z_2$ of ramification type $(2,\ldots, 2)$ and at $z_0$ of type $(d')$. By the Riemann-Hurwitz formula, we have 
$$2d' - 2 = (d' - 1) + \frac{d'}{2} + \frac{d'}{2}, $$
which is impossible. 

If $C'$ maps to $Q_{12}$, it is branched at $z_2$ of ramification type $(2, \ldots, 2)$, at $z_0$ of type $(d')$ and at $z_{1}$ of type 
$(2n_{a_1} + 1, \ldots, 2n_{a_i} + 1, 2, \ldots, 2)$. By Riemann-Hurwitz, we have 
$$2d' - 2 = (d' - 1) + \frac{d'}{2} + \sum_{j=1}^i 2n_{a_j} + \frac{d' - \sum_{j=1}^i (2n_{a_j}+1) }{2}. $$
After simplifying, we get 
$$\sum_{j=1}^{i} (2n_{a_j} - 1) = -2. $$
Since $n_i \geq 1$ in the partition $\nu = (2n_1-1, \ldots, 2n_l-1)$, the last equality is impossible. 
\end{proof}

\begin{corollary}
\label{slopenu}
Let $d_i = 2n_i - 1$ in the partition $\nu = (2n_1-1, \ldots, 2n_l-1)$. Then the slope of a component $\ZZ$ of $\overline{\HH}_d({\bf c_{\nu}})$ equals
$$ s(\ZZ) =   \frac{12}{\displaystyle 1 + \frac{1}{4}\Big(\sum_{i=1}^l \frac{d_i(d_i+4)}{d_i+2}\Big)\frac{|\OO|}{\delta_{\OO}}}. $$
\end{corollary}

\begin{proof}
By Lemma~\ref{notail}, a node of a degenerate cover in $\overline{\HH}_d({\bf c_{\nu}})$ maps to a node of its stabilization. By definition, we have $\delta_{\OO} = \delta'_{\OO}$ for any component $\ZZ$ of $\overline{\HH}_d({\bf c_{\nu}})$. Plugging the ramification profile ${\bf c_{\nu}}$ in 
Theorem~\ref{slope}, we thus obtain the desired expression for $s(\ZZ)$. 
\end{proof}

For the partition $\nu = (d_1,\ldots, d_l)$, where $d_i = 2n_i - 1$, define 
$$\kappa_{\nu} = \displaystyle \frac{1}{24}\sum_{i=1}^l \frac{d_i(d_i+4)}{d_i+2}.$$ 

For the canonical double cover $\tilde C\rightarrow C$, the cohomology $H^{1}(\tilde C; \mathbb R)$ splits into the direct sum 
$H^{1}_+(\tilde C; \mathbb R)\oplus H^{1}_{-}(\tilde C; \mathbb R)$ of the invariant and anti-invariant subspaces under the action 
$\sigma^{*}$, where $\sigma$ is the involution of the double cover. We know that $H^{1}_+(\tilde C; \mathbb R)\cong \sigma^{*}
H^{1}(C; \mathbb R)$ has dimension $2g$. Let $H^1_+$ be the invariant subbundle of the Hodge bundle on the moduli space $\QQ (\nu)$
whose fiber over $(C, \psi)$ is naturally identified as $H^{1}_+(\tilde C; \mathbb R)$. Let $\lambda_1^+\geq \ldots \geq \lambda_g^+$ be the non-negative Laypunov exponents of $H^1_+$ over $\QQ (\nu)$ of quadratic differentials under the Teichm\"{u}ller geodesic flow. Denote their sum by $$L_{\nu} = \lambda_1^+ + \ldots + \lambda_g^+.$$

Let $c_{\nu}$ be the Siegel-Veech area constant of the stratum $\QQ (\nu)$ satisfying 
$c_{\nu} = \frac{\pi^2}{3} c_{area}(\QQ (\nu))$ in the context of \cite{EKZ1}. It measures the average number of 
horizontal cylinders with weight ``height/length'' on a random flat surface parameterized in $\QQ (\nu)$. By \cite[Theorem 2.2]{EKZ1}, 
we know 
$$ L_{\nu} = \kappa_{\nu} + c_{\nu}. $$

Let $N_{d,\nu}= |Cov_d({\bf c_{\nu}})|$ denote the cardinality of the covering set $Cov_d({\bf c_{\nu}})$. Then $N_{d,\nu}$ is equal to the degree of the finite morphism $\overline{\HH}_d({\bf c_{\nu}})\rightarrow \Mof$. Let $\delta_{d, \nu}$ be the sum of $\delta_{\OO}$ ranging over all the orbits $\OO$ 
of $Cov_d({\bf c_{\nu}})$. Recall that the covers in $Cov(\nu)$ correspond to lattice points of $\QQ (\nu)$. During the degeneration of a smooth cover, a horizontal cylinder of height 1 and length $k$ on its square-tiled surface model shrinks to a node of weight $\frac{1}{k}$. By the same argument as in \cite[Appendix A]{C2}, we have 
$$ c_{\nu} = \lim_{d\to \infty} \frac{\delta_{d, \nu}}{6N_{d,\nu}}. $$
The coefficient 6 at the bottom arises by the following reason. First, there are three directions of degeneration 
for a cover in $Cov(\nu)$. To get $\delta_{d,\nu}$, we enumerate all the weighted nodes appearing in the process. But the Siegel-Veech constant only counts those arising in the horizontal direction. Hence, we need to correct the relation by a factor $\frac{1}{3}$. In addition, the orbifold $\PP^1$ as a pillow case is tiled by a white square and a black one. If a node of a degenerate cover is formed by shrinking a horizontal cylinder of height 1 and length $k$ in our setting, it gives rise to a horizontal cylinder consisting of $k$ pairs of squares with height 1 and total length $2k$ in the context of \cite{EKZ1} and \cite{FMZ}. So we need another factor $\frac{1}{2}$. 

The picture below shows $\PP^1$ as a pillow case as well as a horizontal cylinder consisting of $k$ pairs of white and black squares on a covering square-tiled surface. The horizontal vanishing cycle $\alpha$ shrinks to a node that separates $z_1, z_2$ on one rational component and $z_3, z_4$ on the other in the degeneration process. The horizontal cylinder gives rise to a node of weight $\frac{1}{k}$ for enumerating $\delta$. The reader can refer to \cite{FMZ} for a detailed illustration on square-tiled surfaces of this type.  

\begin{figure}[H]
    \centering
    \psfrag{1}{$z_{1}$}
    \psfrag{2}{$z_{2}$}
    \psfrag{3}{$z_{3}$}
    \psfrag{4}{$z_{4}$}
    \psfrag{a}{$1$}
    \psfrag{b}{$k$}
    \psfrag{f}{$\alpha$}
    \includegraphics[scale=0.5]{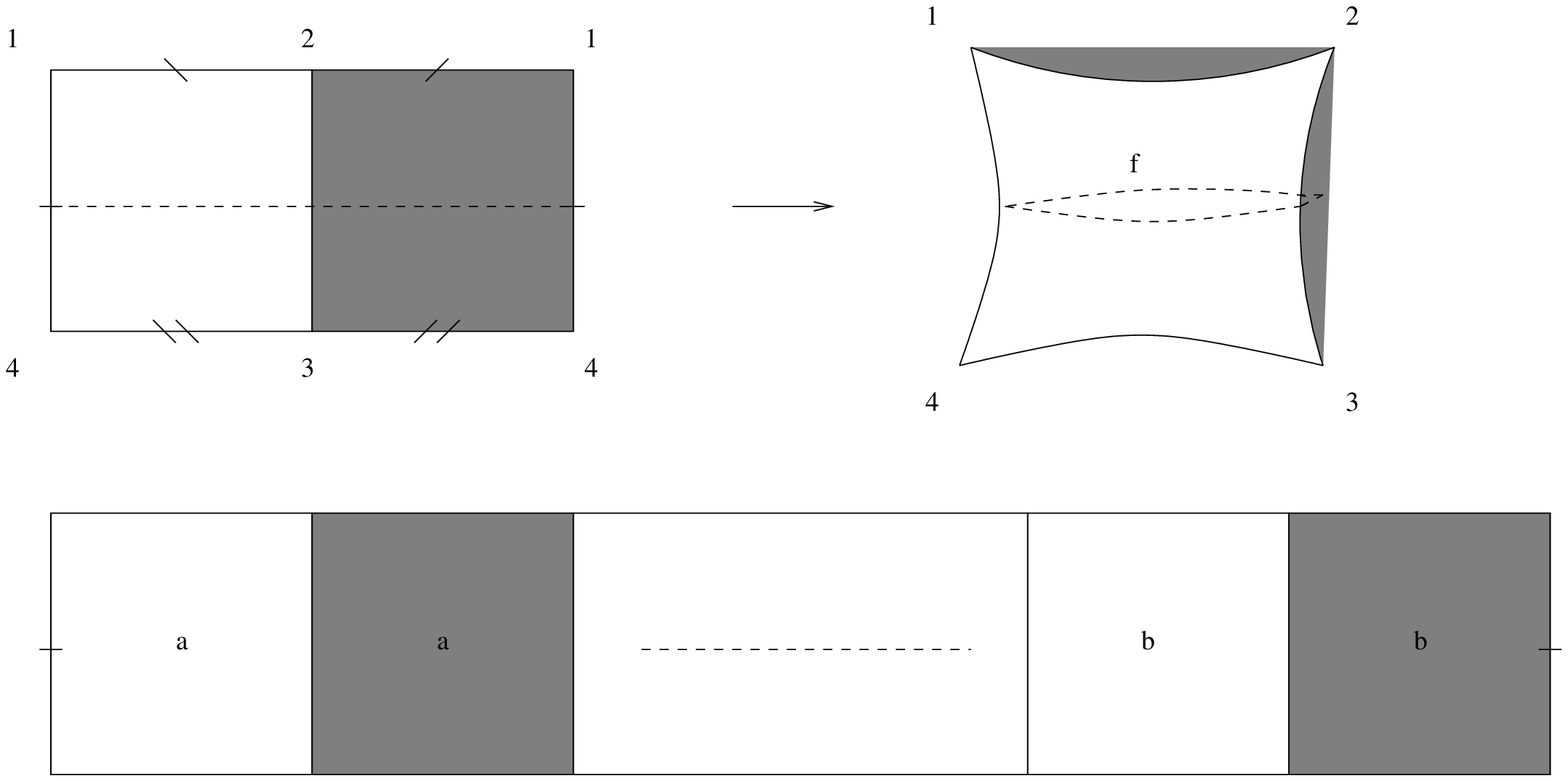}
    \end{figure}

Combining the above results, we derive a relation among the slope, the Siegel-Veech constant and the sum of Lyapunov exponents as follows. 

\begin{theorem}
\label{relation}
For a partition $\nu =  (d_1,\ldots, d_l)$ of $4g-4$ into odd parts, let $s_{\nu}$ be the limit of $s(\overline{\HH}_d({\bf c_{\nu}}))$ 
as $d$ approaches infinity. Then we have $$ s_{\nu} = \frac{12c_{\nu}}{L_{\nu}}. $$  
\end{theorem}

\begin{proof}
Applying Corollary~\ref{slopenu} to $\overline{\HH}_d({\bf c_{\nu}})$, we have 
$$ s(\overline{\HH}_d({\bf c_{\nu}})) = \frac{12}{\displaystyle 1 + \kappa_{\nu}\frac{6N_{d,\nu}}{\delta_{d,\nu}}}. $$
Since $ c_{\nu} = \lim\limits_{d\to \infty} \frac{\delta_{d, \nu}}{6N_{d,\nu}}$ and $ L_{\nu} = \kappa_{\nu} + c_{\nu}$, 
we obtain that 
$$ s_{\nu} = \lim_{d\to\infty} s(\overline{\HH}_d({\bf c_{\nu}})) = \frac{12}{ \displaystyle 1+\frac{\kappa_{\nu}}{c_{\nu}}} = \frac{12c_{\nu}}{L_{\nu}}. $$
\end{proof}

For SL$(2, \RR)$ invariant non-hyperelliptic strata $\tilde{\mathcal M}$ in the moduli space of Abelian differentials, Eskin and Zorich obtained strong numerical evidence, which predicts that their Siegel-Veech area constants $c(\tilde{\mathcal M})$ approach 2 as $g$ goes to infinity. 
Let $\mathcal M$ be an SL$(2, \RR)$ invariant stratum in the moduli space of quadratic differentials. Via the canonical double cover construction, we can associate $\mathcal M$ an invariant stratum $\tilde{\mathcal M}$ in the corresponding space of Abelian differentials. Their Siegel-Veech area constants satisfy the relation 
$$ c(\tilde{\mathcal M}) = 2 c(\mathcal M). $$
Therefore, we expect that 
$$ \lim\limits_{g\to \infty} c_{\nu} = 1. $$
Combining this with Corollary~\ref{bounding} and Theorem~\ref{relation}, a lower bound for slopes of effective divisors on $\Mg$ 
would be arbitrarily close to 
$$ s = \frac{12}{1 + \kappa_{\nu}} = \frac{288}{24+ \displaystyle \sum_{i=1}^l\frac{d_i(d_i+4)}{d_i+2}}$$
for $g \gg 0$, where $(d_1,\ldots, d_l)$ is a partition of $4g-4$ into odd parts and $l \geq g$. The partition 
$(1,\ldots, 1, 3g-3)$ for $l = g$ maximizes the bound as 
$$s\sim \frac{432}{7g} $$
for $g\gg 0$. 

A slightly better lower bound $\sim \frac{576}{5g}$ was first found in \cite{HM1} and recently recovered in \cite{C2}. To the author's best knowledge, all known effective divisors on $\Mg$ have slope bigger than 6. On the other hand, all known lower bounds for slopes of effective divisors grow like $O(1/g)$ as $g$ goes to infinity. It is still an open problem for understanding the effective cone of 
$\Mg$ for large $g$.  

\noindent {\bf Acknowledgements.} I sincerely thank Alex Eskin and Anton Zorich for suggestions and help on cyclic covers, Teichm\"{u}ller curves and Laypunov exponents.

\end{document}